\documentclass[11pt]{article}%
\usepackage{amssymb}
\usepackage{amsfonts}
\usepackage{amsmath}
\usepackage{euler}
\usepackage{graphicx}%
\providecommand{\U}[1]{\protect\rule{.1in}{.1in}}
\newtheorem{theorem}{Theorem}

\newtheorem{corollary}[theorem]{Corollary}

\newtheorem{lemma}[theorem]{Lemma}

\newtheorem{proposition}[theorem]{Proposition}

\newenvironment{proof}[1][Proof]{\noindent\textit{#1.} }{}
\begin{document}

\title{Affine Hulls and Simplices: a Constructive Analysis}
\author{Douglas S. Bridges}
\maketitle

\begin{abstract}%
\noindent
This paper deals with certain fundamental results about affine hulls and
simplices in a real normed linear space. The framework of the paper is
Bishop's constructive mathematics, which, with its characteristic
interpretation of\ \emph{existence }as \emph{constructibility}, often involves
more subtle estimation than its classical-logic-based counterpart. As well as
technically more involved proofs (for example, that of Theorem \ref{sept13t1}
on the perturbation of vertices), we have included a number of elementary ones
for completeness of exposition.

\end{abstract}

%

\setcounter{secnumdepth}{0}%
\normalfont\sf

\section{Introduction}

In this article, which was originally written to facilitate a proof in
\cite{DSBmcom}, we present a constructive\footnote{%
\normalfont\sf
The constructive mathematical framework in which our work is developed is
Bishop's constructive mathematics, \textbf{BISH}
\cite{Bishop,BB,BVtech,Handbook}. In practice, this is mathematics using
intuitionistic logic and an appropriate set- or type-theory, such as are found
in \cite[Chs. 1,2]{Handbook} and \cite{Balps,BMorse,DP}.} development of
fundamentals of the theory of simplices in a real\footnote{%
\normalfont\sf
When we refer to a \emph{real linear/normed space} $X$, we mean that $X$ is a
linear/normed linear space over $\mathbf{R}$.} normed linear space $X$. In
order to do so, we have included a certain amount of basic material that
should be known to the reader; but even for some of that material the proofs,
being fully constructive, may not be entirely standard.

For background material on constructive analysis, we refer the reader to the
bibliography, and in particular to Chapters 8 and 9 of \emph{The Handbook of
Constructive Mathematics \ }\cite{Handbook}. However, we note here one feature
of BISH: if $x,y$ are elements of a normed linear space $X$, then `$x\neq y$'
means `$\left\Vert x-y\right\Vert >0$', which is a constructively stronger
statement than `$\lnot(x=y)$'. In keeping with this, finitely many vectors
$v_{1},\ldots,v_{n}$ in $X$ are said to be \emph{linearly independent} if
$\sum\limits_{i=1}^{n}\lambda_{i}v_{i}\neq0$ for all scalars $\lambda
_{1},\ldots,\lambda_{n}$ such that $\sum\limits_{i=1}^{n}\left\vert
\lambda_{i}\right\vert >0.$ This condition is constructively stronger than,
and implies, the standard classically equivalent defining condition for linear
independence: namely, that if $\lambda_{i}$ $(1\leq i\leq n)\ $are scalars
such that $\sum\limits_{i=1}^{n}\lambda_{i}v_{i}=0,$ then $\lambda_{i}=0$ for
each $i$.%

\medskip
\noindent
\textbf{Notation: }By the \emph{standard basis of} $\mathbf{R}^{n}$ we mean
$\left\{  e_{1},\ldots,e_{n}\right\}  $, where the vector $e_{k}$ has
$k^{\text{\emph{th}}}$ component $1$ and all other components $0$. For
$n\geq2$ we use boldface letters $\mathbf{x},\mathbf{\lambda},\ldots$ to
denote elements of $\mathbf{R}^{n}$ with $k^{th}$ component $x_{k},\lambda
_{k},\ldots$ relative to the standard basis of $\mathbf{R}^{n}$. We define
these two norms on $\mathbf{R}^{n}:$%
\[
\left\Vert \mathbf{x}\right\Vert _{1}\equiv\sum_{k=1}^{n}\left\vert
x_{k}\right\vert \text{ \ and \ }\left\Vert \mathbf{x}\right\Vert _{2}%
\equiv\sqrt{\sum_{k=1}^{n}x_{k}^{2}}\text{.}%
\]%
\medskip

\section{Affine hulls and simplices}

Let $a_{1},\ldots,a_{k}$ be points of a linear space $X$ over $\mathbf{R}$,
and $\lambda_{1},\ldots\lambda_{k}$ real numbers such that $\sum_{i=1}%
^{k}\lambda_{k}=1$. Then $\sum_{i=1}^{k}\lambda_{k}a_{k}$ is an \emph{affine
combination} of the points $a_{k}$. If, in addition, $\lambda_{i}\geq0$ for
each $i$, then $\sum_{i=1}^{k}\lambda_{k}a_{k}$ is a \emph{convex combination
}of those points. If $S\subset X,$ then the set $\mathsf{aff}\ S$of all affine
combinations of elements of $S$ is called the \emph{affine hull }of $S$, and
the set $\mathsf{co}(S)$ of all convex combinations of elements of $S$ is
called the \emph{convex hull} of $S$. For simplicity we often drop parentheses
and write $\mathsf{aff\,}S$ and $\mathsf{co\,}S$

\begin{proposition}
\label{2507pr1}An affine combination of affine combinations in a real linear
space $X$ is an affine combination.
\end{proposition}

\begin{proof}
Consider first an affine combination of two affine combinations.. Let
$x=\sum_{i=1}^{m}\xi_{i}b_{i}$ and $y=\sum_{j=1}^{k}\eta_{j}c_{j}$, where each
$b_{i}$ and each $c_{j}$ belongs to $X$, each $\xi_{i}$ and each $\eta_{j}$
belongs to $\mathbf{R}$, and $\sum_{i=1}^{m}\xi_{i}=1=$ $\sum_{j=1}^{k}%
\eta_{j}$. For each $\alpha\in\mathbf{R},$%
\[
\alpha x+\left(  1-\alpha\right)  y=\sum_{i=1}^{m}\alpha\xi_{i}b_{i}%
+\sum_{j=1}^{k}\left(  1-\alpha\right)  \eta_{j}c_{j}%
\]
where%
\[
\sum_{i=1}^{m}\alpha\xi_{i}+\sum_{j=1}^{k}\left(  1-\alpha\right)  \eta
_{j}=\alpha\sum_{i=1}^{m}\xi_{i}+\left(  1-\alpha\right)  \sum_{j=1}^{k}%
\eta_{j}=1.
\]
Thus $\alpha x+\left(  1-\alpha\right)  y$ is an affine combination of the
points $b_{i}$ and $c_{j}$ $(1\leq i\leq m$, $1\leq j\leq k)$ in $X$.
Iterating this argument, a finite number of times completes the proof.%

\hfill
$\square$
\end{proof}

\begin{proposition}
\label{jul22p1}If $F$ is an inhabited, finitely enumerable subset of a real
linear space $X$, then $\mathsf{aff\ }F=\mathsf{aff(co}\ F)$.
\end{proposition}

\begin{proof}
Let $F=\left\{  a_{1},\ldots,a_{n}\right\}  $ and $x\in\mathsf{aff(co}\ F)$.
Then there exist $b_{j},\xi_{j}\ (1\leq j\leq m)\ $such that $b_{j}%
\in\mathsf{co\ }F,\sum_{j=1}^{m}\xi_{j}=1$, and $x=\sum_{j=1}^{m}\xi_{j}b_{j}%
$. Choose $\beta_{j,k}\geq0$ $\left(  1\leq k\leq n\right)  $\ such that
$\sum_{k=1}^{n}\beta_{j,k}=1$ and $b_{j}=\sum_{k=1}^{n}\beta_{j,k}a_{k}$. Then%
\[
x=\sum_{j=1}^{m}\xi_{j}\sum_{k=1}^{n}\beta_{j,k}a_{k}=\sum_{k=1}^{n}\sum
_{j=1}^{m}\xi_{j}\beta_{j,k}a_{k},
\]
where%
\[
\sum_{k=1}^{n}\sum_{j=1}^{m}\xi_{j}\beta_{j,k}=\sum_{j=1}^{m}\xi_{j}\sum
_{k=1}^{n}\beta_{j,k}=\sum_{j=1}^{m}\xi_{j}=1.
\]
Hence $x\in\mathsf{aff}\left\{  a_{1},\ldots.a_{n+1}\right\}  $. Thus
$\mathsf{aff}(\Sigma)\subset\mathsf{aff}\left\{  a_{1},\ldots,a_{n+1}\right\}
$. The reverse inclusion is trivial.%
\hfill
$\square$
\end{proof}

We say that elements $a_{1},\ldots,a_{n+1}$ of a normed space $X$ are
\emph{affinely independent} if the vectors $a_{k}-a_{n+1}$ $\left(  1\leq
k\leq n\right)  $ are linearly independent. In that case, for any $j$ with
$1\leq j\leq n+1$ the vectors $a_{k}-a_{j}\ \left(  1\leq k\leq n+1,k\neq
j\right)  $ are linearly independent.

\begin{lemma}
\label{sept03L1}Let $a_{1},\ldots,a_{n+1}$ be elements of a normed space $X$,
and let $\mathbf{\lambda},\mathbf{\mu\in}\mathbf{R}^{n+1}$ be such that $%
{\textstyle\sum_{k=1}^{n+1}}
\lambda_{k}=1=%
{\textstyle\sum_{k=1}^{n+1}}
\mu_{k}$. Then%
\[%
{\textstyle\sum_{k=1}^{n+1}}
\lambda_{k}a_{k}-%
{\textstyle\sum_{k=1}^{n+1}}
\mu_{k}a_{k}=%
{\textstyle\sum_{k=1}^{n}}
\left(  \lambda_{k}-\mu_{k}\right)  (a_{k}-a_{n+1}).
\]

\end{lemma}

\begin{proof}
We have%
\begin{align}%
{\textstyle\sum_{k=1}^{n+1}}
\lambda_{k}a_{k}-%
{\textstyle\sum_{k=1}^{n+1}}
\mu_{k}a_{k}  &  =%
{\textstyle\sum_{k=1}^{n}}
\lambda_{k}a_{k}+\left(  1-%
{\textstyle\sum_{k=1}^{n}}
\lambda_{k}\right)  a_{n+1}\nonumber\\
&  -%
{\textstyle\sum_{k=1}^{n}}
\mu_{k}a_{k}-\left(  1-%
{\textstyle\sum_{k=1}^{n}}
\mu_{k}\right)  a_{n+1}\nonumber\\
&  =%
{\textstyle\sum_{k=1}^{n}}
\left(  \lambda_{k}-\mu_{k}\right)  (a_{k}-a_{n+1}). \tag*{$\square$}%
\end{align}

\end{proof}

\begin{proposition}
\label{jul23p1}Let $X$ be a normed space, and $a_{1},\ldots,a_{n+1}$ affinely
independent elements of $X$. Let $\mathbf{\lambda},\mathbf{\mu\in}%
\mathbf{R}^{n+1}$ be such that $\mathbf{\lambda}\neq\mathbf{\mu}$ and $%
{\textstyle\sum_{k=1}^{n+1}}
\lambda_{k}=1=%
{\textstyle\sum_{k=1}^{n+1}}
\mu_{k}$. Then $%
{\textstyle\sum_{k=1}^{n+1}}
\lambda_{k}a_{k}\neq%
{\textstyle\sum_{k=1}^{n+1}}
\mu_{k}a_{k}.$
\end{proposition}

\begin{proof}
Since $\mathbf{\lambda}\neq\mathbf{\mu}$, either $\left\vert \lambda_{k}%
-\mu_{k}\right\vert >0$ for some $k\leq n$, in which case, by Lemma
\ref{sept03L1} and linear independence,%
\[
\left\Vert
{\textstyle\sum_{k=1}^{n+1}}
\lambda_{k}a_{k}-%
{\textstyle\sum_{k=1}^{n+1}}
\mu_{k}a_{k}\right\Vert =\left\Vert
{\textstyle\sum_{k=1}^{n}}
\left(  \lambda_{k}-\mu_{k}\right)  (a_{k}-a_{n+1})\right\Vert >0\text{;}%
\]
or else $\lambda_{n+1}\neq\mu_{n+1}$. In the latter case%
\[%
{\textstyle\sum_{k=1}^{n}}
\left\vert \lambda_{k}-\mu_{k}\right\vert \geq\left\vert
{\textstyle\sum_{k=1}^{n}}
\left(  \lambda_{k}-\mu_{k}\right)  \right\vert =\left\vert (1-\lambda
_{n+1})-(1-\mu_{n+1})\right\vert =\left\vert \lambda_{n+1}-\mu_{n+1}%
\right\vert >0
\]
and we are back in the first case.%
\hfill
$\square$
\end{proof}

\begin{corollary}
\label{jul24c1A}Under the hypotheses of \emph{Proposition \ref{jul23p1}}, for
each $x\in\mathsf{aff}\left\{  a_{1},\ldots,a_{n+1}\right\}  $ there exists a
\emph{unique} $\mathbf{\lambda}\in\mathbf{R}^{n+1}$ such that $%
{\textstyle\sum_{k=1}^{n+1}}
\lambda_{k}=1$ and $x=%
{\textstyle\sum_{k=1}^{n+1}}
\lambda_{k}a_{k}$.
\end{corollary}

\begin{corollary}
\label{jul29c1}Let $X$ be a real normed space, and $a_{1},\ldots,a_{n+1}$
affinely independent elements of $X$. Let $\mathbf{\lambda}\in\mathbf{R}%
^{n+1}$, $%
{\textstyle\sum_{k=1}^{n+1}}
\lambda_{k}=1$, and $x=\sum_{k=1}^{n+1}\lambda_{k}a_{k}$. If $1\leq k\leq n+1$
and $\lambda_{k}\neq1$, then $x\neq a_{k}$.
\end{corollary}

\begin{lemma}
\label{aug04l1A}Let $a_{1},\ldots,a_{n+1}$ be affinely independent elements of
a real normed linear space $X$. Then for each $\varepsilon>0$ there exists
$\delta>0$ with the following properties.

\begin{enumerate}
\item[\emph{(i)}] If%
\begin{equation}
\ \mathbf{\xi},\mathbf{\eta}\in\mathbf{R}^{n+1},\ \sum_{k=1}^{n+1}\xi
_{k}=1=\sum_{k=1}^{n+1}\eta_{k},\ x=\sum_{k=1}^{n+1}\xi_{k}a_{k}%
,\ y=\sum_{k=1}^{n+1}\eta_{k}a_{k}, \label{SS1}%
\end{equation}
$\ $and$\ \left\Vert x-y\right\Vert <\delta,$ then $\sum_{k=1}^{n+1}\left\vert
\xi_{k}-\eta_{k}\right\vert <\varepsilon$.

\item[\emph{(ii)}] If \emph{(\ref{SS1})} holds and $\sum_{k=1}^{n}\left\vert
\xi_{k}-\eta_{k}\right\vert <\delta$, then $\left\Vert x-y\right\Vert
<\varepsilon$.
\end{enumerate}
\end{lemma}

\begin{proof}
Since the vectors in $X$ are linearly independent, it follows from
\cite[4.1.8]{BVtech} that
\[
u:\mathbf{x}\rightsquigarrow\sum_{k=1}^{n}x_{k}\left(  a_{k}-a_{n+1}\right)
\]
is a bounded linear injection of $\mathbf{R}^{n}$ onto the linear space $V$
generated by the linearly independent vectors $a_{k}-a_{n+1}$ $(1\leq k\leq
n)$. Thus for each $\varepsilon>0$ there exists $\delta>0$ such that if
$\mathbf{\zeta}\in\mathbf{R}^{n}$, and $\left\Vert \sum_{k=1}^{n}\zeta
_{k}\left(  a_{k}-a_{n+1}\right)  \right\Vert <\delta$, then $\sum_{k=1}%
^{n}\left\vert \zeta_{k}\right\vert <\varepsilon/2$. Suppose that (\ref{SS1})
holds. By Lemma \ref{sept03L1},%
\[
x-y=\sum_{k=1}^{n}\left(  \xi_{k}-\lambda_{k}\right)  (a_{k}-a_{n+1}),
\]
so if $\left\Vert x-y\right\Vert <\delta$, then $\sum_{k=1}^{n}\left\vert
\xi_{k}-\eta_{k}\right\vert <\varepsilon/2$; moreover, in that case,%
\begin{align*}
\left\vert \xi_{n+1}-\eta_{n+1}\right\vert  &  =\left\vert 1-\sum_{k=1}^{n}%
\xi_{k}-\left(  1-\sum_{k=1}^{n}\eta_{k}\right)  \right\vert \\
&  =\left\vert \sum_{k=1}^{n}\left(  \xi_{k}-\eta_{k}\right)  \right\vert
\leq\sum_{k=1}^{n}\left\vert \xi_{k}-\eta_{k}\right\vert <\varepsilon/2.
\end{align*}
Hence $\sum_{k=1}^{n+1}\left\vert \xi_{k}-\eta_{k}\right\vert <\varepsilon$.
This completes the proof of (i). Since, by \cite[4.1.8]{BVtech}, the inverse
of $u$ is a bounded linear mapping on $V$, we may assume that if
$\mathbf{\zeta}\in\mathbf{R}^{n}$ and $\sum_{k=1}^{n}\left\vert \zeta
_{k}\right\vert <\delta$, then$\ \left\Vert \sum_{k=1}^{n}\zeta_{k}\left(
a_{k}-a_{n+1}\right)  \right\Vert <\varepsilon$. If (\ref{SS1}) holds and
$\sum_{k=1}^{n}\left\vert \xi_{k}-\lambda_{k}\right\vert <\delta$, then, again
using Lemma \ref{sept03L1}, we have
\[
\left\Vert x-y\right\Vert =\left\Vert \sum_{k=1}^{n}\left(  \xi_{k}%
-\lambda_{k}\right)  (a_{k}-a_{n+1})\right\Vert <\varepsilon\text{.}%
\]
This completes the proof of (ii).%
\hfill
$\square$
\end{proof}

\begin{corollary}
\label{sept18c1}Let $a_{1},\ldots,a_{n+1}$ be affinely independent elements of
a real normed linear space $X$. Let $\mathbf{\lambda}\in\mathbf{\Delta}_{n}$,
$\sum_{k=1}^{n+1}\lambda_{k}=1$, and $v=\sum_{k=1}^{n+1}\lambda_{k}a_{k}$. If
$1\leq k\leq n+1$ and $v\neq a_{k}$, then $\lambda_{k}<1$.
\end{corollary}

\begin{proof}
By Lemma \ref{aug04l1A}, there exists $\delta$ with $0<\delta<1$ such that if
(\ref{SS1}) holds and $\sum_{j=1}^{n}\left\vert \xi_{j}-\eta_{j}\right\vert
<\delta$, then $\left\Vert x-y\right\Vert <\left\Vert v-a_{k}\right\Vert $.
Either $\lambda_{k}<1$ or else $\lambda_{k}>1-\delta/2$. In the latter case,
$\sum_{j=1,j\neq k}^{n+1}\lambda_{j}=1-\lambda_{k}<\delta/2$, so $\sum
_{j=1}^{n+1}\left\vert \lambda_{k}-\mu_{k}\right\vert <\delta$, where $\mu
_{k}=0$ $\left(  k\neq j\right)  $ and $\mu_{j}=1$. Hence $\left\Vert
v-a_{j}\right\Vert <\left\Vert v-a_{j}\right\Vert $, which is absurd; from
which it follows that $\lambda_{j}<1$.%
\hfill
$\square$
\end{proof}

%

\smallskip
If $a_{1},\ldots,a_{n+1}$ are affinely independent elements of a real normed
space $X$, then%
\[
\mathsf{co}\left\{  a_{1},a_{1},\ldots,a_{n+1}\right\}  \equiv\left\{
{\textstyle\sum_{i=1}^{n+1}}
\lambda_{i}a_{i}:\lambda_{i}\geq0\ (1\leq i\leq n+1)\text{ and }%
{\textstyle\sum_{i=1}^{n+1}}
\lambda_{i}=1\right\}
\]
is called an $n$\emph{-simplex}\textbf{\ }$\Sigma$\textbf{\ }\emph{with
vertices}\textbf{\ }$a_{1},a_{1},\ldots,a_{n+1}$. It follows from the remark
immediately preceding Proposition \ref{jul23p1} that if $\sigma$ is any
permutation of $\left\{  1,2,\ldots,n+1\right\}  $, then
\[
\mathsf{co}\left\{  a_{\sigma(1),}a_{\sigma(2)},\ldots,a_{\sigma
(n+1)}\right\}  =\Sigma\text{.}%
\]
If $m$ is a nonnegative integer $\leq n$, then the convex hull of any $m+1$
distinct vertices of $\Sigma$ is called an $m$\emph{-face} of $\Sigma$. Each
vertex is therefore a $0$-face. We call $1$-faces \emph{edges}, $\left(
n-1\right)  $-faces \emph{facets} or just \emph{faces}, and%
\[
\mathsf{co}\left\{  a_{j}:1\leq j\leq n+1,j\neq k\right\}
\]
the face of $\Sigma$ \emph{opposite }$a_{k}$. The \emph{barycentre of }%
$\Sigma$ is the point $\sum_{k=1}^{n+1}\frac{1}{n+1}a_{k}$. If
$\mathbf{\lambda}\in\mathbf{R}^{n+1}$, $\lambda_{k}\geq0\ \left(  1\leq k\leq
n+1\right)  $, and $\sum_{k=1}^{n+1}\lambda_{k}=1$, then $\mathbf{\lambda}$ is
called the \emph{barycentric coordinate vector}, and $\lambda_{k}$ the
$k^{\text{\emph{th}}}$ \emph{barycentric coordinate}, of $%
{\textstyle\sum_{k=1}^{n+1}}
\lambda_{k}a_{k}$.

An important example of an $n$- simplex is the \emph{standard }$n$%
\emph{-simplex}\textbf{\ }$\Delta_{n}$,\textbf{\ }in $\mathbf{R}^{n+1}$.
\[
\Delta_{n}\equiv\mathsf{co}\{e_{1},\ldots,e_{n+1}\},
\]
in which the face of $\Delta_{n}$ opposite $e_{n+1}$ is $\Delta_{n-1}$. To
provide an example of an $n$-simplex in $\mathbf{R}^{n}$, let $u\equiv
\sum_{k=1}^{n}e_{k}$,%
\begin{equation}
v_{n+1}=\frac{1}{n}\left(  1+\sqrt{n+1)}\right)  u, \label{dd1}%
\end{equation}
and%
\[
\Sigma_{n}\equiv\mathsf{co}(e_{1},\ldots,e_{n},v_{n+1}).
\]
According to Wikipedia at

\begin{quote}
{\small https://en.wikipedia.org/wiki/Simplex\#Cartesian\_coordinates\_}

{\small for\_a\_regular\_n-dimensional\_simplex\_in\_Rn}
\end{quote}

%

\noindent
$\Sigma_{n}$ is an $n$-simplex in $\mathbf{R}^{n}$ that is\emph{\ regular} in
the sense that for each $m\leq n$ its $m$-faces are congruent (in the sense of
Euclidean geometry). In particular, each of its edges has length $\sqrt{2}$.
Also, the face opposite the vertex $v_{n+1}$ is $\Delta_{n-1}$, and the
barycentre of $\Sigma_{n}$ is%
\begin{align}
b_{n}  &  \equiv\frac{1}{n+1}\left(  \sum_{k=1}^{n}e_{k}+v_{n+1}\right)
\nonumber\\
&  =\frac{1}{n+1}v_{n+1}+\left(  1-\frac{1}{n+1}\right)  \frac{1}{n}u,
\label{ee1}%
\end{align}
which lies in the open segment of $\mathbf{R}^{n}$ with endpoints $v_{n+1}$
and $\frac{1}{n}u$, the barycentre of $\Delta_{n-1}$.

\begin{lemma}
\label{jul17l1}For each positive integer $n$, $\Delta_{n}$ is a complete
subset of $\mathbf{R}^{n+1}$.
\end{lemma}

\begin{proof}
Since all norms on $\mathbf{R}^{n+1}$ are equivalent, we lose no generality in
working with $\left\Vert .\right\Vert _{1}$. Let $\left(  \mathbf{\lambda}%
_{k}\right)  _{k\geq1}$ be a Cauchy sequence in $\Delta_{n}$, where
$\mathbf{\lambda}_{k}=\left(  \lambda_{k,1},\ldots,\lambda_{k,n+1}\right)  $,
$\sum_{m=1}^{n+1}\lambda_{k,j}=1,$ and each $\lambda_{k,j}\geq0$. For each
$\varepsilon>0$ there exists $N$ such that%
\[
\sum_{m=1}^{n+1}\left\vert \lambda_{j,m}-\lambda_{k,m}\right\vert =\left\Vert
\mathbf{\lambda}_{j}-\mathbf{\lambda}_{k}\right\Vert _{1}\leq\varepsilon
\ \ \ \ \ \left(  j,k\geq N\right)  .
\]
It follows that for $1\leq m\leq n+1$, $\left(  \lambda_{k,m}\right)
_{k\geq1}$ is a Cauchy sequence in $\mathbf{R}$ and therefore converges to a
limit $\lambda_{m}\in\mathbf{R}$. Clearly, each $\lambda_{m}\geq0$,
$\sum_{m=1}^{n+1}\lambda_{m}=1$, and $\mathbf{\lambda}_{k}\rightarrow
\mathbf{\lambda}\equiv\sum_{l=1}^{n+1}\lambda_{l}e_{l}$ as $k\rightarrow
\infty$. Thus $\Delta_{n}$ is complete.%
\hfill
$\square$
\end{proof}

\begin{lemma}
\label{jul07l2}For each positive integer $n$, the set
\[
A_{n}\equiv\left\{  (1-\lambda)e_{n+1}+\lambda x:0\leq\lambda\leq
1,\ x\in\Delta_{n-1}\right\}
\]
is dense in $\Delta_{n}$.
\end{lemma}

\begin{proof}
If $v=$ $(1-\lambda)e_{n+1}+\lambda x$ is in $A_{n}$, then there exist
$\lambda_{1},\ldots,\lambda_{n}\geq0$ such that $\sum_{k=1}^{n}\lambda_{k}=1$
and $x=\sum_{k=1}^{n}\lambda_{k}e_{k}$; whence%
\[
v=\left(  1-\lambda\right)  e_{n+1}+\sum_{n=1}^{n}\lambda\lambda_{k}e_{k},
\]
where each $\lambda\lambda_{k}\geq0$ and $\left(  1-\lambda\right)
+\sum_{n=1}^{n}\lambda\lambda_{k}=1$; so $v\in\Delta_{n}$. Hence $A_{n}%
\subset\Delta_{n}$. Now consider $x\equiv\sum_{k=1}^{n+1}\xi_{k}e_{k}\in
\Delta_{n}$, where each $\xi_{k}\geq0$ and $\sum_{k=1}^{n+1}\xi_{k}=1$. If
$x\neq a_{n+1}$, then by Corollary \ref{sept18c1}, $\xi_{n+1}<1$; whence
$\sum_{j=1}^{n}\xi_{j}>0$, \ so
\[
x=\left(  1-\sum_{j=1}^{n}\xi_{j}\right)  e_{n+1}+\sum_{j=1}^{n}\xi_{j}%
\sum_{k=1}^{n}\frac{\xi_{k}}{\sum_{j=1}^{n}\xi_{j}}e_{k},
\]
where $0\leq\sum_{j=1}^{n}\xi_{j}\leq1$,%
\[
\sum_{k=1}^{n}\frac{\xi_{k}}{\sum_{j=1}^{n}\xi_{j}}e_{k}\in\Delta_{n-1},
\]
and therefore $x\in A_{n}$. Given $\varepsilon>0$, we have either $x\neq
a_{n+1}$ and therefore $x\in A_{n}$, or else $\left\Vert x-a_{n+1}\right\Vert
<\varepsilon$, where $a_{n+1}\in A_{n}$. Since $\varepsilon>0$ is arbitrary,
it follows that $A_{n}$ is dense in $\Delta_{n}$.%
\hfill
$\square$
\end{proof}

\begin{proposition}
\label{jul17p1}For each positive integer $n$, $\Delta_{n}$ is compact.
\end{proposition}

\begin{proof}
In view of Lemma \ref{jul17l1}, it suffices to prove, by induction, that
$\Delta_{n}$ is totally bounded. Clearly, $\Delta_{1}$ is totally bounded. Let
$n>1$ and suppose that $\Delta_{n-1}$ is totally bounded. Define the mapping
\[
\phi:\left(  \lambda,v\right)  \rightsquigarrow\left(  1-\lambda\right)
e_{n+1}+\lambda v\ \ \ \left(  \lambda\in\left[  0,1\right]  ,v\in\Delta
_{n-1}\right)  .
\]
We have
\[
\left\Vert \phi(\lambda,v)-\phi(\lambda^{\prime},v^{\prime})\right\Vert
_{1}\leq\left\vert \lambda-\lambda^{\prime}\right\vert +\left\Vert
v-v^{\prime}\right\Vert \ \ \left(  \lambda\in\left[  0,1\right]  ,v\in
\Delta_{n-1}\right)  ,
\]
from which it follows that $\phi\ $is uniformly continuous on the totally
bounded set $\left[  0,1\right]  \times\Delta_{n-1}$. The range of $\phi$ is
$A_{n}$, which is therefore totally bounded. Since, by Lemma \ref{jul07l2},
$A_{n}$ is dense in $\Delta_{n}$, the latter set is also totally bounded
\cite[2.2.6]{BVtech}.%
\hfill
$\square$
\end{proof}

A bijection $f$ of a metric space $X$ onto a metric space $Y$ is
\emph{uniformly bicontinuous }if $f$ and its inverse mapping are uniformly
continuous on $X$ and $Y$, respectively.

\begin{proposition}
\label{may29p1}Let $X$ be a normed space, and $a_{1},\ldots,a_{n+1}$ affinely
independent elements of $X$. Then%
\[
f_{n}:\mathbf{\lambda}\rightsquigarrow\sum_{k=1}^{n+1}\lambda_{k}%
a_{k}\ \ (\mathbf{\lambda}\in\mathsf{aff}(\Delta_{n}))
\]
is a uniformly bicontinuous injection of $\mathsf{aff}(\Delta_{n})$ onto
$\mathsf{aff}\left\{  a_{1},\ldots,a_{n+1}\right\}  $.
\end{proposition}

\begin{proof}
Let $V$ be the $n$-dimensional subspace of $X$ with basis vectors
$a_{k}-a_{n+1}\ (1\leq k\leq n)$. By \cite[4.1.4]{BVtech},%
\[
F:\mathbf{\lambda}\rightsquigarrow\sum_{k=1}^{n}\lambda_{k}\left(
a_{k}-a_{n+1}\right)
\]
is a bounded linear injection of $\mathbf{R}^{n}$ onto $V$ with bounded linear
inverse. Moreover, since both $\mathbf{R}^{n}$ and $V$ are finite dimensional,
the operator norm%
\[
\left\Vert F\right\Vert \equiv\sup\left\{  \left\Vert F(\mathbf{\lambda
})\right\Vert :\mathbf{\lambda}\in\mathbf{R}^{n},\left\Vert \mathbf{\lambda
}\right\Vert _{1}\leq1\right\}
\]
and the counterpart for $F^{-1}$ exist. Now, it is clear that $f_{n}$ maps
$\Delta_{n}$ onto $\Sigma$. For each $\mathbf{\lambda}\in\mathsf{aff}%
(\Delta_{n})$,%
\begin{align*}
f_{n}(\mathbf{\lambda})  &  =%
{\textstyle\sum_{k=1}^{n}}
\lambda_{k}a_{k}+\lambda_{n+1}a_{n+1}\\
&  =%
{\textstyle\sum_{k=1}^{n}}
\lambda_{k}a_{k}+\left(  1-%
{\textstyle\sum_{k=1}^{n}}
\lambda_{k}\right)  a_{n+1}\\
&  =%
{\textstyle\sum_{k=1}^{n}}
\lambda_{k}(a_{k}-a_{n+1})+a_{n+1}\\
&  =F(\lambda_{1},\ldots,\lambda_{n})+a_{n+1}.
\end{align*}
Thus if $\mathbf{\lambda},\mathbf{\mu}\in\mathsf{aff}(\Delta_{n})$, then%
\[
f_{n}(\mathbf{\lambda})-f_{n}(\mathbf{\mu})=F(\lambda_{1},\ldots,\lambda
_{n})-F(\mu_{1},\ldots,\mu_{n})=F(\lambda_{1}-\mu_{1},\ldots,\lambda_{n}%
-\mu_{n}).
\]
Hence%
\begin{align*}
\left\Vert f_{n}(\mathbf{\lambda})-f_{n}(\mathbf{\mu})\right\Vert  &
=\left\Vert F(\lambda_{1}-\mu_{1},\ldots,\lambda_{n}-\mu_{n})\right\Vert \\
&  \leq\left\Vert F\right\Vert \left\Vert (\lambda_{1}-\mu_{1},\ldots
,\lambda_{n}-\mu_{n})\right\Vert _{1}\\
&  =\left\Vert F\right\Vert \sum_{k=1}^{n}\left\vert \lambda_{k}-\mu
_{k}\right\vert \leq\left\Vert F\right\Vert \sum_{k=1}^{n+1}\left\vert
\lambda_{k}-\mu_{k}\right\vert =\left\Vert F\right\Vert \left\Vert
\mathbf{\lambda}-\mathbf{\mu}\right\Vert _{1},
\end{align*}
from which we see that $f_{n}$ is uniformly continuous on $\Delta_{n}$. On the
other hand, by Proposition \ref{jul23p1}, $f_{n}$ is injective and so has
inverse mapping $f_{n}^{-1}$ from $\mathsf{aff}\left\{  a_{1},\ldots
,a_{n+1}\right\}  $ onto $\mathsf{aff}(\Delta_{n})$. If $\mathbf{\lambda
},\mathbf{\mu}\in\mathsf{aff}(\Delta_{n})$, then%
\begin{align*}
\left\Vert \mathbf{\lambda}-\mathbf{\mu}\right\Vert _{1}  &  =%
{\textstyle\sum_{k=1}^{n}}
\left\vert \lambda_{k}-\mu_{k}\right\vert +\left\vert \lambda_{n+1}-\mu
_{n+1}\right\vert \\
&  =%
{\textstyle\sum_{k=1}^{n}}
\left\vert \lambda_{k}-\mu_{k}\right\vert +\left\vert \left(  1-\sum_{k=1}%
^{n}\lambda_{k}\right)  -\left(  1-\sum_{k=1}^{n}\mu_{k}\right)  \right\vert
\\
&  =\sum_{k=1}^{n}\left\vert \lambda_{k}-\mu_{k}\right\vert +\left\vert
\sum_{k=1}^{n}\left(  \lambda_{k}-\mu_{k}\right)  \right\vert \\
&  \leq2\sum_{k=1}^{n}\left\vert \lambda_{k}-\mu_{k}\right\vert \\
&  =2\left\Vert (\lambda_{1}-\mu_{1},\ldots,\lambda_{n}-\mu_{n})\right\Vert
_{1}\\
&  =2\left\Vert F^{-1}\circ F(\lambda_{1}-\mu_{1},\ldots,\lambda_{n}-\mu
_{n})\right\Vert _{1}\\
&  \leq2\left\Vert F^{-1}\right\Vert \left\Vert F(\lambda_{1}-\mu_{1}%
,\ldots,\lambda_{n}-\mu_{n})\right\Vert \\
&  =2\left\Vert F^{-1}\right\Vert \left\Vert f_{n}(\mathbf{\lambda}%
)-f_{n}(\mathbf{\mu})\right\Vert ,
\end{align*}
from which we see that $f_{n}^{-1}$ is uniformly continuous on $\mathsf{aff}%
\left\{  a_{1},\ldots,a_{n+1}\right\}  $.%
\hfill
$\square$
\end{proof}

We call two $n$-simplices $\Sigma$ and $\Sigma^{\prime}$ \emph{equivalent }if
there is a uniformly bicontinuous bijection of\emph{\ }$\Sigma$ onto
$\Sigma^{\prime}$; such a mapping is called an \emph{equivalence }of $\Sigma$
with $\Sigma^{\prime}$.

\begin{corollary}
\label{2507c1}Under the hypotheses of \emph{Proposition \ref{may29p1}}, let
$\Sigma$ be the $n$-simplex with vertices $a_{1},\ldots,a_{n+1}$. Then the
restriction of $f_{n}$ to $\Delta_{n}$ is an equivalence of $\Delta_{n}$ with
$\Sigma$.
\end{corollary}

\begin{corollary}
\label{may30c1}Any two $n$-simplices are equivalent. In fact, if $\Sigma$ is
an $n$-simplex with vertices $a_{1},\ldots.a_{n+1}$, and $\Sigma^{\prime}$ is
an $n$-simplex with vertices $a_{1}^{\prime},\ldots,a_{n+1}^{\prime}$, then
the mapping $\sum_{k=1}^{n+1}\lambda_{k}a_{k}\rightsquigarrow\sum_{k=1}%
^{n+1}\lambda_{k}a_{k}^{\prime}$ is an equivalence of $\Sigma$ with
$\Sigma^{\prime}$.
\end{corollary}

\begin{proof}
This is a straightforward consequence of Corollary \ref{2507c1}\emph{.}%
\hfill
$\square$
\end{proof}

\begin{corollary}
\label{jul17c1}Every $n$-simplex is compact.
\end{corollary}

\begin{proof}
Let $\Sigma$ be an $n$-simplex with vertices $a_{1},\ldots,a_{n+1}$ in a
normed space $X$. In the notation of Proposition \ref{may29p1}, since
$\Delta_{n}$ is compact (Proposition \ref{jul17p1}) and $f_{n}$ is uniformly
continuous, we see from \cite[2.2.6]{BVtech} that $f_{n}(\Delta_{n})$---that
is, $\Sigma$---is totally bounded. On the other hand, from Lemma \ref{jul17l1}
and a simple application of the uniform continuity of the functions
$f_{n}^{-1}$ and $f_{n}$ it is simple to show that $\Sigma$ is complete and
hence compact.%
\hfill
$\square$
\end{proof}

\begin{corollary}
\label{jul23c1}Let $\Sigma$ be an $n$-simplex $\Sigma$ with vertices
$a_{1},\ldots,a_{n+1}$ in a real normed linear space $X$, let $1\leq\nu\leq
n+1$, and let $F$ be the face of $\Sigma$ opposite to $a_{\nu}$. Let
$x=\sum_{k=1}^{n+1}\lambda_{k}a_{k}\in\Sigma$, where $\mathbf{\lambda}%
\in\Delta_{n}$. Then $\rho(x,F)>0$ if and only if $\lambda_{\nu}>0$.
\end{corollary}

\begin{proof}
Without loss of generality, we may assume that $\nu=n+1$. Being an $\left(
n-1\right)  $-simplex, $F$ is compact (by Corollary \ref{jul17c1}) and hence
both complete and located. By \cite[3.3.1]{BVtech}, there exists $\mu\in
\Delta_{n-1}$ such that if $x\neq y\equiv\sum_{k=1}^{n}\mu_{j}a_{k}$, then
$\rho(x,F)>0$. If $\lambda_{n+1}>0$, then Proposition \ref{jul23p1} shows that
$x\neq y$, so $\rho(x,F)>0$. Conversely, if $\rho(x,F)>0$, then%
\[
0<\left\Vert x-\sum_{k=1}^{n}\lambda_{k}a_{k}\right\Vert =\lambda
_{n+1}\left\Vert a_{n+1}\right\Vert ,
\]
so $\lambda_{n+1}>0$.%
\hfill
$\square$
\end{proof}

\begin{lemma}
\label{aug1}Let $\mathbf{\lambda}\in\Delta_{n}$ and $\mathbf{\mu}\in
\mathbf{R}^{n+1}$. Then%
\[
\sum_{k=1}^{n}\left\vert \mu_{k}(1-\lambda_{n+1})-\lambda_{k}(1-\mu
_{n+1})\right\vert \leq2\sum_{k=1}^{n+1}\left\vert \mu_{k}-\lambda
_{k}\right\vert .
\]

\end{lemma}

\begin{proof}
For $1\leq k\leq n$ we have%
\begin{align*}
\left\vert \mu_{k}(1-\lambda_{n+1})-\lambda_{k}(1-\mu_{n+1})\right\vert  &
=\left\vert \mu_{k}-\lambda_{k}+\lambda_{n+1}(\lambda_{k}-\mu_{k})+\lambda
_{k}(\mu_{n+1}-\lambda_{n+1})\right\vert \\
&  \leq(1+\lambda_{n+1})\left\vert \mu_{k}-\lambda_{k}\right\vert +\lambda
_{k}\left\vert \mu_{n+1}-\lambda_{n+1}\right\vert \\
&  \leq2\left\vert \mu_{k}-\lambda_{k}\right\vert +\left\vert \mu
_{n+1}-\lambda_{n+1}\right\vert \lambda_{k}.
\end{align*}
Hence%
\begin{align}
\sum_{k=1}^{n}\left\vert \mu_{k}(1-\lambda_{n+1})-\lambda_{k}(1-\mu
_{n+1})\right\vert  &  \leq\sum_{k=1}^{n}2\left\vert \mu_{k}-\lambda
_{k}\right\vert +\left\vert \mu_{n+1}-\lambda_{n+1}\right\vert \sum_{k=1}%
^{n}\lambda_{k}\nonumber\\
&  \leq\sum_{k=1}^{n}2\left\vert \mu_{k}-\lambda_{k}\right\vert +\left\vert
\mu_{n+1}-\lambda_{n+1}\right\vert \nonumber\\
&  \leq2\sum_{k=1}^{n+1}\left\vert \mu_{k}-\lambda_{k}\right\vert .
\tag*{$\square$}%
\end{align}

\end{proof}

\begin{lemma}
\label{aug2A}Let $\mathbf{\lambda}\in\Delta_{n}$ be such that $\lambda
_{n+1}<1$, let $0<\theta<1$, and let $\mathbf{\mu}\in\mathbf{R}^{n+1}$ be such
that $\left\vert 1-\mu_{n+1}\right\vert \geq\theta\left(  1-\lambda
_{n+1}\right)  $. Then%
\[
\sum_{k=1}^{n}\left\vert \frac{\mu_{k}}{1-\mu_{n+1}}-\frac{\lambda_{k}%
}{1-\lambda_{n+1}}\right\vert \leq\frac{2}{\theta\left(  1-\lambda
_{n+1}\right)  ^{2}}\sum_{k=1}^{n+1}\left\vert \mu_{k}-\lambda_{k}\right\vert
.
\]

\end{lemma}

\begin{proof}
By Lemma \ref{aug1},%
\begin{align}
\sum_{k=1}^{n}\left\vert \frac{\mu_{k}}{1-\mu_{n+1}}-\frac{\lambda_{k}%
}{1-\lambda_{n+1}}\right\vert  &  =\sum_{k=1}^{n}\left\vert \frac{\mu
_{k}(1-\lambda_{n+1})-\lambda_{k}(1-\mu_{n+1})}{(1-\mu_{n+1})(1-\lambda
_{n+1})}\right\vert \nonumber\\
&  =\sum_{k=1}^{n}\frac{\left\vert \mu_{k}(1-\lambda_{n+1})-\lambda_{k}%
(1-\mu_{n+1})\right\vert }{\left\vert 1-\mu_{n+1}\right\vert \left(
1-\lambda_{n+1}\right)  }\nonumber\\
&  =\frac{\sum_{k=1}^{n}\left\vert \mu_{k}(1-\lambda_{n+1})-\lambda_{k}%
(1-\mu_{n+1})\right\vert }{\left\vert 1-\mu_{n+1}\right\vert \left(
1-\lambda_{n+1}\right)  }\nonumber\\
&  \leq\frac{2}{\theta\left(  1-\lambda_{n+1}\right)  ^{2}}\sum_{k=1}%
^{n+1}\left\vert \mu_{k}-\lambda_{k}\right\vert . \tag*{$\square$}%
\end{align}

\end{proof}

We define%
\[
\Delta_{n}^{+}\equiv\left\{  \mathbf{\lambda}\in\Delta_{n}:\lambda_{k}>0\text{
}(0\leq k\leq n+1)\right\}  .
\]

\begin{proposition}
\label{jul31l1}Let $a_{1},\ldots,a_{n+1}$ be affinely independent elements of
a real normed linear space $X$, and let $F=\mathsf{co}\left\{  a_{1}%
,\ldots,a_{n}\right\}  .$ Let $c=\sum_{k=1}^{n+1}\lambda_{k}a_{k}$, where
$\mathbf{\lambda}\in\Delta_{n}$ and $0<\lambda_{n+1}<1$. Then%
\[
b\equiv\sum_{k=1}^{n}\frac{\lambda_{k}}{1-\lambda_{n+1}}a_{k}%
\]
is the unique point in which the line $L:a_{n+1}+\mathbf{R}(c-a_{n+1})$ meets
$F$. For each $r>0$ there exists $\delta>0$ such that if $\mathbf{\xi}%
\in\mathbf{R}^{n+1}$, $\sum_{k=1}^{n+1}\xi_{k}=1$, $x=\sum_{k=1}^{n+1}\xi
_{k}a_{k}$, and $\left\Vert x-c\right\Vert <\delta$, then $0<\xi_{n+1}<1$ and%
\begin{equation}
\sum_{k=1}^{n}\frac{\xi_{k}}{1-\xi_{n+1}}a_{k}\in B_{X}(b,r)\cap
\mathsf{aff}(F). \label{La4}%
\end{equation}

\end{proposition}

\begin{proof}
Since $\lambda_{n+1}=1-\sum_{k=1}^{n}\lambda_{k}<1,$ $\frac{\lambda_{k}%
}{1-\lambda_{n+1}}\geq0$ for each $k\leq n$, and $\sum_{k=1}^{n}\frac
{\lambda_{k}}{1-\lambda_{n+1}}=1$, $b$ is well defined and belongs to $F$. On
the other hand, if the line $L$ meets $F$ at the point%
\begin{align*}
\left(  1-t\right)  a_{n+1}+tc  &  =\left(  1-t\right)  a_{n+1}+\sum
_{k=1}^{n+1}t\lambda_{k}a_{k}\\
&  =\left(  1-t+t\lambda_{n+1}\right)  a_{n+1}+\sum_{k=1}^{n}t\lambda_{k}%
a_{k},
\end{align*}
then $1-t+t\lambda_{n+1}=0$, $t=1/(1-\lambda_{n+1})$, and therefore the point
in question equals $b$.

Now let $r>0$,%
\[
M=\max_{1\leq k\leq n+1}\left\Vert a_{k}\right\Vert
\]
and%
\[
\varepsilon=\frac{1}{2}\min\left\{  \lambda_{n+1},1-\lambda_{n+1}%
,\frac{\left(  1-\lambda_{n+1}\right)  ^{2}r}{2M}\right\}  .
\]
By Lemma \ref{aug04l1A}, there exists $\delta>0$ such that if$\ \mathbf{\xi
}\in\mathbf{R}^{n+1}$, $\sum_{k=1}^{n+1}\xi_{k}=1$, $x=\sum_{k=1}^{n+1}\xi
_{k}a_{k}$, and $\left\Vert x-c\right\Vert <\delta$, then $\sum_{k=1}%
^{n+1}\left\vert \xi_{k}-\lambda_{k}\right\vert <\varepsilon$. For such
$\mathbf{\xi}$ and $x$ we have $\left\vert \xi_{n+1}-\lambda_{n+1}\right\vert
<\frac{1}{2}\min\left\{  \lambda_{n+1},1-\lambda_{n+1}\right\}  $,
so$\ \xi_{n+1}>\frac{1}{2}\lambda_{n+1}>0$ and%
\[
\xi_{n+1}<\lambda_{n+1}+\frac{1}{2}\left(  1-\lambda_{n+1}\right)  =\frac
{1}{2}(1+\lambda_{n+1})<1.
\]
Hence%
\[
\left\vert 1-\xi_{n+1}\right\vert =1-\xi_{n+1}>1-\frac{1}{2}(1+\lambda
_{n+1})=\frac{1}{2}(1-\lambda_{n+1})\text{.}%
\]
Since also $\sum_{k=1}^{n}\frac{\xi_{k}}{1-\xi_{n+1}}=1$, we see that
$\sum_{k=1}^{n}\frac{\xi_{k}}{1-\xi_{n+1}}a_{k}\in\mathsf{aff}(F)$. On the
other hand, by Lemma \ref{aug2A},%
\begin{align*}
\left\Vert \sum_{k=1}^{n}\frac{\xi_{k}}{1-\xi_{n+1}}a_{k}-b\right\Vert  &
=\left\Vert \sum_{k=1}^{n}\left(  \frac{\xi_{k}}{1-\xi_{n+1}}-\frac
{\lambda_{k}}{1-\lambda_{n+1}}\right)  a_{k}\right\Vert \\
&  \leq\sum_{k=1}^{n}\left\vert \frac{\xi_{k}}{1-\xi_{n+1}}-\frac{\lambda_{k}%
}{1-\lambda_{n+1}}\right\vert \left\Vert a_{k}\right\Vert \\
&  \leq M\sum_{k=1}^{n}\left\vert \frac{\xi_{k}}{1-\xi_{n+1}}-\frac
{\lambda_{k}}{1-\lambda_{n+1}}\right\vert \\
&  \leq\frac{2M}{\frac{1}{2}\left(  1-\lambda_{n+1}\right)  ^{2}}\sum
_{k=1}^{n+1}\left\vert \xi_{k}-\lambda_{k}\right\vert \\
&  <\frac{4M\varepsilon}{\left(  1-\lambda_{n+1}\right)  ^{2}}<r,
\end{align*}
so (\ref{La4}) holds.%
\hfill
$\square$
\end{proof}

\section{Relative interior points}

If $X$ is a real normed linear space, then the \emph{relative interior},
$\mathsf{relint}(S)$, of $S\subset X$ is the interior of $S$ relative to
$\mathsf{aff}(S)$: thus%
\[
\mathsf{relint}(S)=\left\{  x\in S:\exists r>0\left(  B_{X}(x,r)\cap
\mathsf{aff}(S)\subset S\right)  \right\}  .
\]
Points of $\mathsf{relint}(S)$ are called \emph{relative interior points of
}$S$ \emph{in }$X$.

\begin{corollary}
\label{aug01c1}Let $\Sigma$ be an $n$-simplex with vertices $a_{1}%
,\ldots,a_{n+1}\ $in a real normed linear space $X$, let $F=\mathsf{co}%
\left\{  a_{1},\ldots,a_{n}\right\}  $, and let $c=\sum_{k=1}^{n+1}\lambda
_{k}a_{k}$, where $\mathbf{\lambda}\in\Delta_{n}$ and $0<\lambda_{n+1}<1$. If%
\[
b\equiv\sum_{k=1}^{n}\frac{\lambda_{k}}{1-\lambda_{n+1}}a_{k}%
\]
is a relative interior point of $F$, then $c$ is a relative interior point of
$\Sigma$.
\end{corollary}

\begin{proof}
If $r>0$ and $B_{X}(b,r)\cap\mathsf{aff}(F)\subset F$, construct $\delta$ as
in Proposition \ref{jul31l1}. Let $\mathbf{\xi}\in\mathbf{R}^{n+1}$ and%
\[
x\equiv\sum_{k=1}^{n+1}\xi_{k}a_{k}\in B_{X}(c,\delta)\cap\mathsf{aff}%
(\Sigma).
\]
By our choice of $\delta$, $0<\xi_{n+1}<1$ and%
\[
\sum_{k=1}^{n}\frac{\xi_{k}}{1-\xi_{n+1}}a_{k}\in B_{X}(b,r)\cap
\mathsf{aff}(F)\subset F\subset\Sigma.
\]
Thus%
\[
x=\left(  1-\xi_{n+1}\right)  \sum_{k=1}^{n}\frac{\xi_{k}}{1-\xi_{n+1}}%
a_{k}+\xi_{n+1}a_{n+1},
\]
which is a convex combination of two elements of $\Sigma$ and therefore
belongs to $\Sigma$. It follows that $B_{X}(c,\delta)\cap\mathsf{aff}%
(\Sigma)\subset\Sigma$ and hence that $c$ is a relative interior point of
$\Sigma$.%
\hfill
$\square$
\end{proof}

\begin{proposition}
\label{sept09p1}Let $\Sigma$ be an $n$-simplex in an $n$-dimensional real
normed linear space $X$. Then $\mathsf{relint}(\Sigma)\ $is the interior of
$\Sigma$.
\end{proposition}

\begin{proof}
We may assume that $X=\mathbf{R}^{n}$, so that $\mathsf{aff}(\Sigma
)=\mathbf{R}^{n}$. Then $c$ is a relative interior point of $\Sigma$ if and
only if there exists $r>0$ such that%
\[
B_{\mathbf{R}^{n}}(c,r)=B_{\mathbf{R}^{n}}(c,r)\cap\mathbf{R}^{n}%
=B_{\mathbf{R}^{n}}(c,r)\cap\mathsf{aff}(\Sigma)\subset\Sigma
\]
---that is, $c$ belongs to the interior of $\Sigma$.%
\hfill
$\square$
\end{proof}

We now provide alternative characterisations of relative interior points of a simplex.

\begin{proposition}
\label{jul21p3}Let $\Sigma$ be an $n$-simplex with vertices $a_{1}%
,\ldots,a_{n+1}$ in a real normed linear space $X$, let $\mathbf{\lambda}%
\in\Delta_{n}$, and let $c=\sum_{k=1}^{n}\lambda_{k}a_{k}\in\Sigma$. Then the
following conditions are equivalent

\begin{enumerate}
\item[\emph{(i)}] $c$ is a relative interior point of $\Sigma$ in $X.$

\item[\emph{(ii)}] $c$ is bounded away from each $(n-1)$-face of $\Sigma$.

\item[\emph{(iii)}] $c=\sum_{k=1}^{n+1}\lambda_{k}a_{k}$, where
$\mathbf{\lambda}\in\Delta_{n}^{+}$.
\end{enumerate}
\end{proposition}

\begin{proof}
For each $k\leq n+1$ let $F_{k}$ be the face of $\Sigma$ opposite $a_{k}$, and
note that by Corollary \ref{jul23c1}, $\rho(a_{k},F_{k})>0$. If (i) holds,
then there exists $r>0$ such that $B_{X}(c,r)\cap\mathsf{aff}(\Sigma
)\subset\Sigma$. Let $\nu\leq n+1$ and $z\in F_{\nu}$, and assume $\left\Vert
z-c\right\Vert <r/2$. There exists $\mathbf{\zeta}\in\Delta_{n} $ with
$\zeta_{\nu}=0$, $\sum_{k=1}^{n+1}\zeta_{k}=1$, and $z=\sum_{k=1}^{n+1}%
\zeta_{k}a_{k}$. Choose $\theta>1$ such that $\left(  \theta-1\right)
\left\Vert a_{\nu}-c\right\Vert <r/2$ and $\theta\left\Vert z-c\right\Vert
<r/2$, and let $y=\left(  1-\theta\right)  a_{\nu}+\theta z\in\mathsf{aff}%
(\Sigma)$. Then%
\[
\left\Vert y-c\right\Vert \leq\left(  \theta-1\right)  \left\Vert a_{\nu
}-c\right\Vert +\theta\left\Vert z-c\right\Vert <r,
\]
so $y\in B_{X}(c,r)\cap\mathsf{aff}(\Sigma)$ and therefore $y\in\Sigma$. Hence
there exists $\mathbf{\eta}\in\Delta_{n}$ such that $y=\sum_{k=1}^{n+1}%
\eta_{k}a_{k}$. But $y=\left(  1-\theta\right)  a_{n+1}+\sum_{k=1}^{n}%
\theta\zeta_{k}a_{k}$,\ where $\left(  1-\theta\right)  +\sum_{k=1}^{n}%
\theta\zeta_{k}=1$, so by Corollary \ref{jul24c1A}, $\eta_{n+1}=1-\theta<0$, a
contradiction from which it follows that $\left\Vert z-c\right\Vert \geq r/2$.
Since $z\in F_{\nu}$ is arbitrary, we conclude that $\rho(c,F_{\nu})\geq r/2$.
Hence (i) implies (ii). \label{000022}The equivalence of (ii) and (iii) is a
simple consequence of Corollary \ref{jul23c1}.

Finally, assuming (iii), we prove by induction on $n$ that (i) holds. If $n=1
$, then there exists $\lambda$ such that $0<\lambda<1$ and $c=\lambda
a_{1}+(1-\lambda)a_{2}$. Let%
\[
r=\frac{1}{2}\left\Vert a_{1}-a_{2}\right\Vert \min\left\{  \lambda
,1-\lambda\right\}  >0.
\]
If%
\[
x=\xi a_{1}+\left(  1-\xi\right)  a_{2}\in B_{X}(c,r)\cap\mathsf{aff}%
(\Sigma),
\]
then%
\begin{align*}
\left\vert \xi-\lambda\right\vert  &  =\frac{1}{\left\Vert a_{1}%
-a_{2}\right\Vert }\left\vert \xi-\lambda\right\vert \left\Vert a_{1}%
-a_{2}\right\Vert \\
&  =\frac{1}{\left\Vert a_{1}-a_{2}\right\Vert }\left\Vert x-c\right\Vert
<\frac{1}{2}\min\left\{  \lambda,1-\lambda\right\}  ,
\end{align*}
so $0<\xi<1$ and therefore $x\in\Sigma$. This proves the case $n=1$. Now let
$n>1$ and suppose that we have proved that if $v_{1},\ldots,v_{n}$ are
affinely independent vectors in $X$ and $\mathsf{\alpha}\in\Delta_{n-1}^{+}$,
then $\sum_{k=1}^{n}\alpha_{k}v_{k}$ is a relative interior point of the
$(n-1)$-simplex with vertices $v_{1},\ldots,v_{n}$. Consider the case where
$\Sigma$ is an $n$-simplex. By Corollary \ref{jul29c1}, $c\neq a_{n+1}$. Let
\[
b\equiv\sum_{k=1}^{n}\frac{\lambda_{k}}{1-\lambda_{n+1}}a_{k}%
\]
By our induction hypothesis, $b$ is a relative interior point of the
$(n-1)$-simplex $F_{n+1}$, so there exists $r>0$ such that $B_{X}%
(b,r)\cap\mathsf{aff(}F_{n+1})\subset F_{n+1}$. By Proposition \ref{jul31l1},
there exists $\delta>0$ such that if $\mathbf{\eta\in R}^{n+1}$, $\sum
_{k=1}^{n+1}\eta_{k}=1$, $y=\sum_{k=1}^{n+1}\eta_{k}a_{k}$, and $\left\Vert
y-c\right\Vert <\delta$, then $0<\eta_{n+1}<1$ and%
\[
\sum_{k=1}^{n}\frac{\eta_{k}}{1-\eta_{n+1}}a_{k}\in B_{X}(b,r)\cap
\Sigma\subset F_{n+1}.
\]
In that case,%
\[
y=\eta_{n+1}a_{n+1}+\left(  1-\eta_{n+1}\right)  \sum_{k=1}^{n}\frac{\eta_{k}%
}{1-\eta_{n+1}}a_{k},
\]
which belongs to $\Sigma$. Thus $B_{X}(c,\delta)\cap\mathsf{aff}%
(\Sigma)\subset\Sigma$. This completes the proof that (iii) implies (i).%
\hfill
$\square$
\end{proof}

\begin{corollary}
\label{jul24c1}The barycentre of a simplex in a real normed linear space is a
relative interior point.
\end{corollary}

\begin{corollary}
\label{sept09c1}Let $\Sigma$ be an $n$-simplex with vertices $a_{1}%
,\ldots,a_{n+1}$ in $\mathbf{R}^{n}$. Then $\Sigma^{\circ}$ is inhabited.
\end{corollary}

\begin{proof}
In view of Propositions \ref{jul21p3} and \ref{sept09p1}, $\frac{1}{n+1}%
\sum_{k=1}^{n+1}a_{k}\in\mathsf{relint}(\Sigma)=\Sigma^{\circ}$.%
\hfill
$\square$
\end{proof}

\begin{corollary}
\label{sept10c1}Under the hypotheses of \emph{Proposition \ref{jul31l1}, }$b$
is a relative interior point of $F$ if and only if $c$ is a relative interior
point of $\Sigma$.
\end{corollary}

\begin{proof}
In view of Corollary \ref{aug01c1} and Proposition \ref{jul21p3}, it suffices
to note that if $c$ is a relative interior point of $\Sigma$, then
$\frac{\lambda_{k}}{1-\lambda_{n+1}}>0$ for each $k\leq n$.%
\hfill
$\square$
\end{proof}

\section{Vertex perturbations}

The next three results will enable us to prove that simplex structure is
preserved under small perturbations of the vertices.

\begin{lemma}
\label{apr28l2}Let $x_{1},\ldots,x_{n}$ be linearly independent vectors in a
real normed linear space $X$. There exists $\delta>0$ such that if $y_{k}\in
X$ and $\left\Vert x_{k}-y_{k}\right\Vert <\delta$ for each $k$, then the
vectors $y_{1},\ldots,y_{n}$ are linearly independent.
\end{lemma}

\begin{proof}
Let $V$ be the $n$-dimensional subspace of $X$ with basis $\left\{
x_{1},\ldots,x_{n}\right\}  $. Since all norms on $V$ are equivalent, there
exists $c>0$ such that\label{zz}%
\[
\left\Vert \sum_{k=1}^{n}\lambda_{k}x_{k}\right\Vert \geq c\sum_{k=1}%
^{n}\left\vert \lambda_{k}\right\vert
\]
for all $\mathbf{\lambda}\in$ $\mathbf{R}^{n}$. Let $\delta=c/(2n+1)$. For
$1\leq k\leq n+1$ let $y_{k}\in X\ $and $\left\Vert y_{k}-x_{k}\right\Vert
<\delta$. If $\lambda_{k}\in\mathbf{R}$ $(1\leq k\leq n)$ and $\sum_{k=1}%
^{n}\left\vert \lambda_{k}\right\vert >0$, then%
\begin{align*}
\left\Vert \sum_{k=1}^{n}\lambda_{k}y_{k}\right\Vert  &  \geq\left\Vert
\sum_{k=1}^{n}\lambda_{k}x_{k}\right\Vert -\left\Vert \sum_{k=1}^{n}%
\lambda_{k}(y_{k}-x_{k})\right\Vert \\
&  \geq c\sum_{k=1}^{n}\left\vert \lambda_{k}\right\vert -\sum_{k=1}%
^{n}\left\vert \lambda_{k}\right\vert \left\Vert (y_{k}-x_{k})\right\Vert \\
&  \geq\left(  2n+1\right)  \delta\sum_{k=1}^{n}\left\vert \lambda
_{k}\right\vert -\sum_{k=1}^{n}\left\vert \lambda_{k}\right\vert \delta\\
&  =2n\delta\sum_{k=1}^{n}\left\vert \lambda_{k}\right\vert >0\text{. }%
\end{align*}
Thus the vectors $y_{k}$ $(1\leq k\leq n)$ are linearly independent.%
\hfill
$\square$
\end{proof}

\begin{corollary}
\label{sept04c1}Let $a_{1},\ldots,a_{n+1}$ be affinely independent elements of
a real normed linear space $X$. Then there exists $\delta>0$ such that if
$x_{1},\ldots,x_{n+1}\ $belong to $X$ and $\left\Vert a_{k}-x_{k}\right\Vert
<\delta$ for each $k\leq n+1$, then the vectors $x_{1},\ldots,x_{n+1}$ are
affinely independent.
\end{corollary}

\begin{proof}
By Lemma \ref{apr28l2}, there exists $\delta>0$ such that if $y_{k}\in X$ and
$\left\Vert \left(  a_{k}-a_{n+1}\right)  -y_{k}\right\Vert <2\delta$ for each
$k\leq n$, then $y_{1},\ldots,y_{n}$ are linearly independent. If $x_{k}\in X$
and$\ \left\Vert a_{k}-x_{k}\right\Vert <\delta$ for each $k\leq n+1$, then%
\[
\left\Vert \left(  a_{k}-a_{n+1}\right)  -\left(  x_{k}-x_{n+1}\right)
\right\Vert \leq\left\Vert a_{k}-x_{k}\right\Vert +\left\Vert a_{n+1}%
-x_{n+1}\right\Vert <2\delta\ \ \ (1\leq k\leq n),
\]
from which the result follows.%
\hfill
$\square$
\end{proof}

%

\medskip
We need some elementary facts about square matrices. We define two norms on
the linear space $M_{n\text{ }}$of all $n$-by-$n$ matrices over $\mathbf{R}$
as follows:%
\begin{align*}
\left\Vert A\right\Vert _{1} &  \equiv\max_{1\leq i,j\leq n}\left\vert
a_{ij}\right\vert ,\\
\left\Vert A\right\Vert _{0} &  \equiv\sup\left\{  \left\Vert A\mathbf{\lambda
}\right\Vert _{1}:\mathbf{\lambda\in R}^{n},\left\Vert \mathbf{\lambda
}\right\Vert _{1}\leq1\right\}  ,
\end{align*}
where $a_{ij}$ is the $\left(  i,j\right)  ^{\mathsf{th}}$ entry of the matrix
$A$, and $\left\Vert \mathbf{\lambda}\right\Vert _{1}=\sum_{k=1}^{n}\left\vert
\lambda_{k}\right\vert $. Since $M_{n}$ is finite-dimensional, there exists
$\kappa>0$ such that $\left\Vert A\right\Vert _{0}\leq\kappa\left\Vert
A\right\Vert _{1}$ for all $A\in M_{n}$.

\begin{lemma}
\label{sept12l1}For each $\varepsilon>0$ there exists $\delta>0$ such that if
$A\equiv\left[  \xi_{kj}\right]  $ is an $n$-by-$n$ matrix over $\mathbf{R}$
and $\left\Vert A-I\right\Vert _{1}<\delta$, then $\det A>1/2\,\ $and
$\left\Vert A^{-1}-I\right\Vert _{1}<\varepsilon$.
\end{lemma}

\begin{proof}
The mapping $\det$ is continuous at $I$, and the mapping carrying $A$ to its
adjoint is continuous everywhere. Thus the mapping%
\[
A\rightsquigarrow A^{-1}=\frac{1}{\det A}\mathsf{adj}A
\]
is continuous at $I$. Since the set of all $n$-by-$n$ matrices over
$\mathbf{R}$ can be regarded as a linear subspace of $\mathbf{R}^{n^{2}}$, on
which all norms are equivalent, the result follows.%
\hfill
$\square$
\end{proof}

\begin{theorem}
\label{sept13t1}Let $X$ be a real normed linear space, $\Sigma$ an $n$-simplex
with vertices $a_{1},\ldots,a_{n+1}$ in $X$, and $c\ $a relative interior
point of $\Sigma$. There exists $\delta>0$ such that if $x_{k}\in
\mathsf{aff}(\Sigma)$ and $\left\Vert x_{k}-a_{k}\right\Vert <\delta$ for each
$k\leq n+1$, then the points $x_{1},\ldots,x_{n+1}$ are affinely independent,
and$\ c$ is a relative interior point of the $n$-simplex $\mathsf{co}\left\{
x_{1},\ldots,x_{n+1}\right\}  $.
\end{theorem}

\begin{proof}
By Corollary \ref{sept04c1}, there exists $\delta_{1}$ with $0<\delta_{1}<1/2$
such that if $x_{1},\ldots,x_{n+1}\ $belong to $X$ and $\left\Vert a_{k}%
-x_{k}\right\Vert <\delta_{1}$ for each $k\leq n+1$, then $x_{1}%
,\ldots,x_{n+1}$ are affinely independent. By Proposition \ref{jul21p3},
$c=\sum_{k=1}^{n+1}\lambda_{k}a_{k}$, where $\mathbf{\lambda}\in\Delta_{n}%
^{+}$. Let%
\begin{equation}
m=\min_{1\leq k\leq n+1}\lambda_{k}.\nonumber
\end{equation}
With $\kappa$ as in the sentence immediately before Lemma \ref{sept12l1},
choose $\delta_{2}$ with $0<\delta_{2}<\delta_{1}$ such that if $A$ is an
$n$-by-$n$ matrix over $\mathbf{R}$ and $\left\Vert A-I\right\Vert _{1}%
<\delta_{2}$, then $\det A>1/2$, $\left\Vert A^{-1}-I\right\Vert
_{1}<m/2\kappa$, and therefore $\left\Vert A^{-1}-I\right\Vert _{0}<m/2$.
Using Lemma \ref{aug04l1A}, construct $\delta$ with $0<\delta<\delta_{2}$ such
that if%
\[
\ \mathbf{\xi},\mathbf{\eta}\in\mathbf{R}^{n+1},\ \sum_{k=1}^{n+1}\xi
_{k}=1=\sum_{k=1}^{n+1}\eta_{k},\ x=\sum_{k=1}^{n+1}\xi_{k}a_{k}%
,\ y=\sum_{k=1}^{n+1}\eta_{k}a_{k},
\]
and $\left\Vert \sum_{k=1}^{n+1}\left(  \xi_{k}-\eta_{k}\right)
a_{k}\right\Vert $ then $\left\vert \xi_{k}-\eta_{k}\right\vert <\delta_{2}$
for each $k\leq n+1$. Consider points $x_{1},\ldots,x_{n+1}$ of $\mathsf{aff}%
(\Sigma)$ such that $\left\Vert x_{k}-a_{k}\right\Vert <\delta$ for each
$k\leq n+1$. Since $\delta<\delta_{1}$, those points $x_{k}$ are affinely
independent, so $\Sigma^{\prime}\equiv\mathsf{co}\left\{  x_{1},\ldots
,x_{n+1}\right\}  $ is an $n$-simplex. For each $k$ there exists a unique
$\mathbf{\xi}_{k}\equiv\left(  \xi_{k,1},\xi_{k,2},\ldots,\xi_{k,n+1}\right)
\in\mathbf{R}^{n+1}$ such that $\sum_{j=1}^{n+1}\xi_{k,j}=1$ and $x_{k}%
=\sum_{j=1}^{n+1}\xi_{k,j}a_{j}$. Then%
\begin{equation}
\left\vert \xi_{k,k}-1\right\vert <\delta_{2}\text{ and }\left\vert \xi
_{k,j}\right\vert <\delta_{2}\text{ for }j\neq k. \label{B0}%
\end{equation}
I claim that the equations%
\begin{equation}
\sum_{k=1}^{n+1}\gamma_{k}\xi_{k,j}=\lambda_{j}\ \ \ \ (1\leq j\leq n+1).
\label{B1}%
\end{equation}
\label{000001}have a unique solution $\mathbf{\gamma}$, that $\mathbf{\gamma
}\in\Delta_{n}^{+}$, and that $c=\sum_{j=1}^{n+1}\gamma_{j}x_{j}$. To justify
this claim, letting $\Xi$ be the $n$-by-$n$ matrix with $\left(  j,k\right)
^{\mathsf{th}}$ component $\xi_{j,k}$, we see that the equations can be
written as $\Xi^{T}\mathbf{\gamma}=\mathbf{\lambda}$, where $^{T}$ denotes
transpose. From (\ref{B0}) we see that $\left\Vert \Xi^{T}-I\right\Vert
_{1}<\delta_{2}$, so
\begin{equation}
\det\Xi^{T}>\frac{1}{2}\text{ and }\left\Vert \left(  \Xi^{T}\right)
^{-1}-I\right\Vert _{0}<\frac{m}{2}. \label{B-1}%
\end{equation}
The equations (\ref{B1}) therefore have a unique solution $\mathbf{\gamma
}=\left(  \Xi^{T}\right)  ^{-1}\mathbf{\lambda}$. Then%
\begin{align*}
c  &  =\sum_{j=1}^{n+1}\lambda_{j}a_{j}=\sum_{j=1}^{n+1}\sum_{k=1}^{n+1}%
\gamma_{k}\xi_{k,j}a_{j}\\
&  =\sum_{k=1}^{n+1}\sum_{j=1}^{n+1}\gamma_{k}\xi_{k,j}a_{j}=\sum_{k=1}%
^{n+1}\gamma_{k}\sum_{j=1}^{n+1}\xi_{k,j}a_{j}=\sum_{k=1}^{n+1}\gamma_{k}x_{k}%
\end{align*}
and%
\begin{align*}
\sum_{k=1}^{n+1}\gamma_{k}  &  =\sum_{k=1}^{n+1}\gamma_{k}1=\sum_{k=1}%
^{n+1}\gamma_{k}\sum_{j=1}^{n+1}\xi_{k,j}=\sum_{k=1}^{n+1}\sum_{j=1}%
^{n+1}\gamma_{k}\xi_{k,j}\\
&  =\sum_{j=1}^{n+1}\sum_{k=1}^{n+1}\gamma_{k}\xi_{k,j}=\sum_{j=1}%
^{n+1}\lambda_{j}=1\text{,}%
\end{align*}
from which we see that $c\in\mathsf{aff}\left\{  x_{1},\ldots,x_{n+1}\right\}
$. On the other hand, by (\ref{B-1}),
\begin{align*}
\sum_{j=1}^{n+1}\left\vert \gamma_{j}-\lambda_{j}\right\vert  &  =\left\Vert
\left(  \Xi^{T}\right)  ^{-1}\mathbf{\lambda-\lambda}\right\Vert
_{1}=\left\Vert \left(  \left(  \Xi^{T}\right)  ^{-1}-I\right)
\mathbf{\lambda}\right\Vert _{1}\\
&  \leq\left\Vert \left(  \Xi^{T}\right)  ^{-1}-I\right\Vert _{0}\left\Vert
\mathbf{\lambda}\right\Vert _{1}=\left\Vert \left(  \Xi^{T}\right)
^{-1}-I\right\Vert _{0}<\frac{m}{2},
\end{align*}
so for each $j\leq n+1$,%
\[
\gamma_{j}>\lambda_{j}-\frac{m}{2}\geq m-\frac{m}{2}>0\text{.}%
\]
Since also $\sum_{j=1}^{n+1}\gamma_{j}=1$, $\mathbf{\gamma}\in\Delta_{n}^{+}$
and therefore, by Proposition \ref{jul21p3}(iii), $c$ is a relative interior
point $\Sigma^{\prime}$.%
\hfill
$\square$
\end{proof}

\begin{corollary}
\label{sept16c1}Let $\Sigma$ be an $n$-simplex with vertices $a_{1}%
,\ldots,a_{n+1}$ in an $n$-dimensional real Banach space, and $c\ $an interior
point of $\Sigma$. There exists $\delta>0$ such that if $x_{k}\in X$ and
$\left\Vert x_{k}-a_{k}\right\Vert <\delta$ for each $k\leq n+1$, then the
convex hull of the points $x_{1},\ldots,x_{n+1}$ is an $n$-simplex with $c$ in
its interior.
\end{corollary}

\begin{proof}
This follows from Theorem \ref{sept13t1} and Proposition \ref{sept09p1}.%
\hfill
$\square$
\end{proof}

\label{e}This corollary is used in the proof of \cite[Proposition 8]{DSBmcom}.%

\bigskip
%

\bigskip

%

\bigskip
%

\bigskip
%

\noindent
\textbf{Author's address: \ }School of Mathematics and Statistics, University
of Canterbury, Christchurch 8140, New Zealand%
\hfill
\textbf{email}: dsbridges.math@gmail.com


\begin{thebibliography}{9}                                                                                                %


\bibitem {Balps}R.A. Alps and D.S. Bridges, \emph{Constructive Morse Set
Theory---a Foundation for Constructive Mathematics}, in preparation.

\bibitem {Bishop}E. Bishop, \emph{Foundations of Constructive Analysis,
}McGraw-Hill, New York, 1967.

\bibitem {BB}E. Bishop and D.S. Bridges, \emph{Constructive Analysis},
Grundlehren der math. Wissenschaften \textbf{279}, Springer Verlag,
Heidelberg-Berlin-New York, 1985.

\bibitem {DSBmcom}D.S. Bridges, \emph{Metric Double Complements of Convex
Sets}, preprint, September 2025.

\bibitem {BMorse}D.S. Bridges, `Morse set theory as a foundation for
constructive mathematics', Theoretical Comp. Sci. 928, 115-135, 2022. https://doi.org/10.1016/j.tcs.2022.06.019

\bibitem {BVtech}D.S. Bridges and L.S. V\^{\i}\c{t}\u{a}, \emph{Techniques of
Constructive Analysis}, Universitext, Springer New York, 2006.

\bibitem {Handbook}D.S. Bridges, H. Ishihara, M.J. Rathjen, H. Schwichtenberg
(editors),\emph{\ Handbook of Constructive Mathematics}, Encyclopedia of
Mathematics and Its Applications \textbf{185}, Cambridge University Press, 2023.

\bibitem {DP}P. Dybjer and E. Palmgren, `Intuitionistic Type Theory', in:
\emph{The Stanford Encyclopedia of Philosophy} (Winter 2024 Edition), Edward
N. Zalta \& Uri Nodelman (eds.), URL https://plato.stanford.edu/archives/win2024/entries/type-theory-intuitionistic.
\end{thebibliography}
\end{document}